\documentclass[12pt]{amsart}

\usepackage{amsmath,amssymb,amsfonts,amsthm}

\usepackage[alphabetic]{amsrefs}
\usepackage{bm}
\usepackage{graphicx}
\usepackage{ascmac}
\usepackage[all]{xy}
\usepackage{mathtools}
\usepackage{tikz-cd}
\usepackage{extarrows}

\usepackage[colorlinks, linkcolor=blue,anchorcolor=blue,citecolor=green]{hyperref}
\usepackage{graphics,epic}

\usepackage{times}
\usepackage{tikz}
\usetikzlibrary{positioning,arrows.meta}
\usetikzlibrary{matrix,calc}
\usetikzlibrary{decorations.pathmorphing}

\newtheorem{theorem}{Theorem}[section]
\newtheorem*{theorem*}{Theorem}

\newtheorem{lemma}[theorem]{Lemma}
\newtheorem{proposition}[theorem]{Proposition}
\newtheorem{corollary}[theorem]{Corollary}

\newtheorem*{conjecture*}{Conjecture}

\newtheorem{remark}[theorem]{Remark}
\newtheorem{definition}[theorem]{Definition}

\newcommand{\leftleadsto}{\mathrel{\reflectbox{$\leadsto$}}}

\newcommand{\opname}[1]{\operatorname{\mathsf{#1}}}

\newcommand{\der}{\cd}

\newcommand{\coker}{\opname{coker}}

\newcommand{\im}{\opname{im}\nolimits}
\renewcommand{\ker}{\opname{ker}\nolimits}

\newcommand{\id}{\mathbf{1}}

\newcommand{\Hom}{\opname{Hom}}

\newcommand{\Ext}{\opname{Ext}}

\newcommand{\cd}{{\mathcal D}}

\begin{document}
	
\title{Homological rigidity of quiver representations over $\mathbb{F}_1$}
\author{Changjian Fu}
\author{Liang Yang}
\author{Zhiyuan Zeng}
\address{Department of Mathematics, Sichuan University, Chengdu, 610064 PR China}

\email{changjianfu@scu.edu.cn(Fu)}
\email{malyang@scu.edu.cn(Yang)}	
\email{2697413645@qq.com(Zeng)}

\begin{abstract}
We establish a homological rigidity phenomenon for the category of representations of quivers over the virtual field $\mathbb{F}_1$, which is inherently non-additive and does not admit classical homological algebra tools. We prove that all higher Yoneda extension groups vanish beyond degree two for arbitrary quivers, including infinite ones. Consequently, the global dimension of the category is universally bounded by 2. Moreover, we obtain a complete classification of quivers according to their homological dimension, which is determined solely by the underlying orientation structure.

\end{abstract}
\maketitle

\section{Introduction}
Let $\mathbb{F}_1$ be the virtual field, i.e., the ``field with one element''. The philosophy of $\mathbb{F}_1$ dates back to Tits \cite{Tits56}, who observed that incidence geometries over a finite field $\mathbb{F}_q$ have  combinatorial counterparts that can be interpreted as incidence
geometries  over $\mathbb{F}_1$. Later, it also appears in Manin's work \cite{Ma95} on translating the geometric proof of
the Weil conjectures from function fields to the case of $\mathbb{Q}$ with the hope to shed some light on the
Riemann Hypothesis. 
To this day, the philosophy of $\mathbb{F}_1$  continues to attract immense interest and has undergone significant development across various branches, including algebraic geometry, number theory, and combinatorics, see  \cites{CC10,CC11,De05,De08,TV09} for instance.

In \cite{Sz12}, Szczesny introduced quiver representations over $\mathbb{F}_1$ as a degenerate, combinatorial analogue of classical quiver representations over fields. Over the past decade, by extending techniques from classical representation theory, researchers have extensively explored the interplay between the combinatorics of $\mathbb{F}_1$-representations and their field-theoretic counterparts (see, e.g., \cites{JKS23,JS23,JS24,Kleinau25}). By contrast, the homological properties of quiver representations over $\mathbb{F}_1$ remain largely unexplored.

Classically, the homological algebra of representation categories is heavily governed by additive and abelian structures. In settings such as quiver representations over a field, projective resolutions and derived categories provide powerful, standard machinery to study homological invariants. In the non-additive setting of $\mathbb{F}_1$-representations, however, these foundational tools are no longer directly available. The resulting categories exhibit fundamentally different homological behaviors, and very little is known about the structures of their higher extension groups.

To package this exploration rigorously, let $Q$ be a finite quiver and $\opname{rep}(Q,\mathbb{F}_1)$ the category of finite-dimensional $\mathbb{F}_1$-representations of $Q$. Although $\opname{rep}(Q,\mathbb{F}_1)$ is not additive, it forms a proto-exact category in the sense of \cite{DK19}. In particular, one can adapt Yoneda's original construction to define the extension spaces $\Ext^n(\mathbf{L},\mathbf{N})$ for any $\mathbf{L},\mathbf{N}\in \opname{rep}(Q,\mathbb{F}_1)$ and any positive integer $n$. Note that $\Ext^n(\mathbf{L},\mathbf{N})$ is not an abelian group but rather a pointed set, or equivalently, an $\mathbb{F}_1$-vector space.

\begin{definition}
Let $Q$ be a quiver. The {\em global dimension} of $\opname{rep}(Q,\mathbb{F}_1)$ is defined as
\[
\opname{gldim} \opname{rep}(Q,\mathbb{F}_1)=\sup\{n \mid \Ext^{n}(-,-)\neq 0\}.
\]
\end{definition}
Similarly, the subcategory $\opname{rep}(Q,\mathbb{F}_1)_{\rm nil}$ consisting of nilpotent $\mathbb{F}_1$-representations is extension-closed. One can thus define the global dimension of $\opname{rep}(Q,\mathbb{F}_1)_{\rm nil}$ in an analogous manner.

Our interest in this global dimension originates from an expectation of Szczesny \cite{Sz12}. Inspired by the fundamental role played by the Euler form in Ringel's theory of Hall algebras \cite{R90}, Szczesny anticipated that the Euler form would play a parallel role in the Hall algebra of $\operatorname{rep}(Q,\mathbb{F}_1)$, provided it is well-defined. However, as noted above, the breakdown of classical homological methods poses a formidable challenge to explicitly computing these $\operatorname{Ext}$-spaces for general quivers. Indeed, it has been  recently observed that the homological properties of $\opname{rep}(Q,\mathbb{F}_1)$ differ drastically from those over fields \cite{FRY24}. Specifically, for a linear quiver of type $\mathbb{A}_n$ with $n\geq 3$, the global dimension is $2$, contrasting sharply with the classical global dimension of $1$. This striking phenomenon naturally led to the conjecture that the global dimension of a general quiver might grow arbitrarily large (e.g., taking values such as $3, 4$, or $5$).

The main objective of this paper is to show that this is surprisingly not the case: the global dimension of an arbitrary quiver (even if infinite) is always bounded by $2$.

\begin{theorem}\label{thm:main-thm-1}
Let $Q$ be an arbitrary quiver. Then we have \[\opname{gldim}\opname{rep}(Q,\mathbb{F}_1)\leq 2\quad \text{and} \quad \opname{gldim}\opname{rep}(Q,\mathbb{F}_1)_{\rm nil}\leq 2.\]
\end{theorem}

The proof of Theorem \ref{thm:main-thm-1} is based on a non-additive version of a well-known result in the derived category of a quiver over a field (cf. Section \ref{s:upper-bound}). The key ingredient is to replace the classical octahedral axiom with an explicit combinatorial lifting construction. 

According to Theorem \ref{thm:main-thm-1}, the global dimension of $\opname{rep}(Q,\mathbb{F}_1)$ can only take values in $\{0, 1, 2\}$. This restriction naturally motivates a classification of quivers according to their global dimensions. By employing the restriction and induction functors for subquivers, we obtain a complete classification basing on its global dimension.

\begin{theorem}\label{thm:main-thm-2}
Let $Q$ be a connected quiver. 
\begin{itemize}
    \item[(1)] $\opname{gldim}\opname{rep}(Q,\mathbb{F}_1)=0$ if and only if $Q$ consists of a single vertex.
    \item[(2)] $\opname{gldim}\opname{rep}(Q,\mathbb{F}_1)=1$ if and only if $Q$ is bipartite but not a single vertex.
    \item[(3)] $\opname{gldim}\opname{rep}(Q,\mathbb{F}_1)=2$ if and only if $Q$ is non-bipartite, equivalently, $Q$ admits either an oriented cycle or a subquiver $\xymatrix{\cdot\ar[r]&\cdot\ar[r]&\cdot}$.
\end{itemize}
\end{theorem}

The paper is organized as follows. In Section \ref{s:preliminary}, we give a brief introduction on quiver representations over the virtual filed $\mathbb{F}_1$ and Yoneda's construction of extension spaces. Section \ref{s:upper-bound} devotes to proving Theorem \ref{thm:main-thm-1}. In Section \ref{s:classification}, we introduce the embedding functor and the restriction functor for subquivers, which enables us to establish a relationship between extension spaces of subquivers and the whole quiver.
 In particular, Theorem \ref{thm:main-thm-2} (3) is proved basing the embedding functor. We present an elementary construction in the bipartite cases in subsection \ref{ss:bipartie-case}, and finish the proof of Theorem \ref{thm:main-thm-2} in subsection \ref{ss:proof-thm-2}.

\section{Preliminary}\label{s:preliminary}
\subsection{Quiver representations over \texorpdfstring{$\mathbb{F}_1$}{F1}}
 We follow \cites{Sz12, JS23} and refer to \cite{Sz12} for unexplained definitions related to representations of quivers over $\mathbb{F}_1$.

A finite-dimensional {\it $\mathbb{F}_1$-vector space} is a finite pointed set $V:=(V,0_V)$. The dimension of $V$ is defined as $\dim V:=|V|-1$. An {\it $\mathbb{F}_1$-linear map} from $V:=(V,0_V)$ to $W:=(W,0_W)$  is a pointed function $f:V\to W$  such that $f|_{V\backslash f^{-1}(0_W)}$ is an injection.   For an $\mathbb{F}_1$-linear map $f:V\to W$, we denote by $f^t: W\to V$ the {\it duality} of $f$, which is the pointed function defined by $f^t(w)=f^{-1}(w)$ for $0_W\neq w\in \opname{im} f=f(V)$ and $f^t(w)=0_V$ otherwise.

Let $Q=(Q_0,Q_1,s,t)$ be a (finite) quiver. A {\it representation} of $Q$ over $\mathbb{F}_1$ is a collection \[\mathbf{M}=(M_i, M_\alpha)_{i\in Q_0,\alpha\in Q_1},\] where
\begin{itemize}
    \item $M_i$ is a finite-dimensional $\mathbb{F}_1$-vector space for each vertex $i\in Q_0$;
    \item $M_\alpha:M_{s(\alpha)}\to M_{t(\alpha)}$ is an $\mathbb{F}_1$-linear map for each arrow $\alpha\in Q_1$.
\end{itemize}
A representation $\mathbf{M}$ is {\em finite-dimensional} if $\dim \mathbf{M}:=\sum_{i\in Q_0}\dim M_i<\infty$. If $Q$ is a finite quiver, then every representation is a finite-dimensional representation by our convention.
A representation $\mathbf{M}=(M_i,M_\alpha)$ is {\it nilpotent}, if there exists $N\geq 0$ such that for any $n\geq N$, $M_{\alpha_n}\circ\cdots\circ M_{\alpha_1}=0$ for any path $\alpha_n\alpha_{n-1}\cdots \alpha_1$ in $Q$.

Let $\mathbf{M}=(M_i,M_\alpha)$ and $\mathbf{N}=(N_i,N_\alpha)$ be representations of $Q$ over $\mathbb{F}_1$. A morphism $\Phi:\mathbf{M}\to \mathbf{N}$ is a collection of $\mathbb{F}_1$-linear maps $\Phi=(\phi_i)_{i\in Q_0}$, where $\phi_i:M_i\to N_i$, such that the following diagram commutes for each arrow $\alpha:i\to j\in Q_1$:
\[
\xymatrix{M_i\ar[r]^{M_\alpha}\ar[d]^{\phi_i}&M_j\ar[d]^{\phi_j}\\
 N_i\ar[r]^{N_\alpha}&N_j.}
\]
The kernel $\ker \Phi$ and image $\opname{im} \Phi$ of $\Phi$ can be defined as usual, which are subrepresentations of $\mathbf{M}$ and $\mathbf{N}$ respectively (cf. \cite{Sz12}*{Section 4}).

Let $\opname{rep}(Q,\mathbb{F}_1)$ be the category of representations of $Q$ over $\mathbb{F}_1$ and $\opname{rep}(Q,\mathbb{F}_1)_{\rm nil}$ the full subcategory of $\opname{rep}(Q,\mathbb{F}_1)$ consisting of nilpotent representations. The category $\opname{rep}(Q,\mathbb{F}_1)$ is not additive, but it shares many good properties with the category of representations of $Q$ over a field. For instance, every morphism has a kernel and a cokernel. It is proto-exact in the sense of Dyckerhoff and Kapranov \cite{DK19}, where all monomorphisms and epimorphisms are admissible.  Since the category is normal, this structure is equivalent to being a proto-abelian category in the sense of \cite{A09}. It is also combinatorial with exact direct sum in the sense of \cite{ELY22}.
Both the
Jordan--H\"{o}lder and Krull--Schmidt theorems hold in $\opname{rep}(Q,\mathbb{F}_1)$; we refer to \cite{Sz12}*{Section 4} for more details. 

\subsection{Coefficient quivers}
Let $\mathbf{V}=(V_i,V_\alpha)_{i\in Q_0,\alpha\in Q_1}\in \opname{rep}(Q,\mathbb{F}_1)$ be a representation. The {\em coefficient quiver} $\Gamma_{\mathbf{V}}$ \cite{JS23} of $\mathbf{V}$ is the quiver defined as follows:
\begin{itemize}
    \item The vertex set is $\bigsqcup_{i\in Q_0} V_i\backslash\{0\}$;
    \item There is an arrow $v\to w$ if there exists an arrow $\alpha\in Q_1$ such that $V_\alpha(v)=w$. In this case, it follows that $v\in V_{s(\alpha)}$ and $w\in V_{t(\alpha)}$.
\end{itemize}
The following properties were established in \cite{JS23}*{Lemma 3.7}.
\begin{lemma}\label{lem:coefficient-indecomposable}
    Let $\mathbf{V}\in \opname{rep}(Q,\mathbb{F}_1)$. The following statement holds:
    \begin{itemize}
        \item[(1)] $\mathbf{V}$ is indecomposable if and only if $\Gamma_\mathbf{V}$ is connected.
        \item[(2)] $\mathbf{V}$ is nilpotent if and only if $\Gamma_\mathbf{V}$ is acyclic.
    \end{itemize}
\end{lemma}
As an application of the coefficient quiver, we have the following result.
\begin{lemma}\label{lem:cyclic-non-nilpotent-indecom}
    Let $Q$ be a cyclic quiver with vertex set $\{1, \dots, n\}$ and arrow set $\{\alpha_i : i \to i+1\}$ (with indices taken modulo $n$). Let $\mathbf{M} \in \opname{rep}(Q, \mathbb{F}_1)$ be an indecomposable, non-nilpotent representation. Then $\Gamma_\mathbf{M}$ is a cyclic quiver with $mn$ vertices for some positive integer $m$, and $\mathbf{M}$ is simple.
\end{lemma}
\begin{proof}
    Since $Q$ is a cyclic quiver, the assertion that $\Gamma_\mathbf{M}$ is also a cyclic quiver follows directly from the definition of the coefficient quiver and Lemma~\ref{lem:coefficient-indecomposable}.
    
    To show that $\mathbf{M}$ is simple, let $\mathbf{N}$ be a nonzero subrepresentation of $\mathbf{M}$. Without loss of generality, we may assume there exists a nonzero element $x \in N_1 \subseteq M_1$. For any path $p = \alpha_k \cdots \alpha_2 \alpha_1$ starting at vertex $1$, we have 
\[
M_p(x) := M_{\alpha_k} \cdots M_{\alpha_2} M_{\alpha_1}(x) \neq 0.
\]
Since $\Gamma_\mathbf{M}$ is a cyclic quiver, the set of all such evaluations \[\{M_p(x) \mid p \text{ is a path starting at vertex } 1\}\] exhausts all nonzero elements in the vertices of $\Gamma_\mathbf{M}$. Because $\mathbf{N}$ is a subrepresentation containing $x$, it must contain all these elements. It follows that $\mathbf{N} = \mathbf{M}$, completing the proof.
\end{proof}
The following is a direct consequence of Lemma \ref{lem:cyclic-non-nilpotent-indecom}.
\begin{corollary}\label{cor:vanish-hom}
    Let $Q$ be a cyclic quiver. Let $\mathbf{M} \in \opname{rep}(Q, \mathbb{F}_1)$ be an indecomposable, non-nilpotent representation, and let $\mathbf{N} \in \opname{rep}(Q, \mathbb{F}_1)$ be an indecomposable, nilpotent representation. Then $\Hom(\mathbf{M}, \mathbf{N}) = 0 = \Hom(\mathbf{N}, \mathbf{M})$.
\end{corollary}

\subsection{Yoneda's Extension} 
Although the category $\opname{rep}(Q,\mathbb{F}_1)$  is not additive, one may apply Yoneda's construction to define the extension space $\operatorname{Ext}^i(\mathbf{N},\mathbf{L})$ for any $\mathbf{L},\mathbf{N}\in \opname{rep}(Q,\mathbb{F}_1)$.
In this section, we recall Yoneda's extension for $\opname{rep}(Q,\mathbb{F}_1)$, following \cite{FRY24}.

Denote by $\mathbf{0}$ the zero object of $\opname{rep}(Q,\mathbb{F}_1)$.
An exact sequence of length $n+2$ of $\opname{rep}(Q,\mathbb{F}_1)$ is a sequence of morphism
\[
\xymatrix{\mathbf{L}\ \ar[r]^{\alpha_0}&\mathbf{M}_1\ar[r]^{\alpha_1}&\mathbf{M}_2\ar[r]^{\alpha_2}&\cdots\ar[r]^{\alpha_{n-1}}&\mathbf{M}_n\ar[r]^{\alpha_n}&\mathbf{N}}
\]
 such that 
 \begin{itemize}
     \item $\ker \alpha_{i+1}=\opname{im} \alpha_i$ for $0\leq i\leq n-1$;
     \item $\alpha_0$ is a monomorphism, equivalently, $\ker\alpha_0=\mathbf{0}$;
     \item $\alpha_n$ is an epimorphism, equivalently, $\opname{im} \alpha_n=\mathbf{N}$.
 \end{itemize}  We will use $\rightarrowtail$ (resp. $\twoheadrightarrow$) to indicate a morphism is a monomorphism  (resp. an epimorphism).


Let $\mathbb{E}^n(\mathbf{N},\mathbf{L})$ be the set of exact sequences of length $n+2$ that start at $\mathbf{L}$ and end at $\mathbf{N}$. An element in $\mathbb{E}^n(\mathbf{N},\mathbf{L})$ is called an {\it $n$-extension} of $\mathbf{N}$ by $\mathbf{L}$. 
Recall that the $n$-extensions $\epsilon$ and $\epsilon'$ of $\mathbf{N}$ by $\mathbf{L}$ satisfy the relation $\epsilon\leadsto\epsilon'$ if there is a commutative diagram
\[
\xymatrix{\epsilon: &\mathbf{L}\ \ar@{=}[d]\ar@{>->}[r]&\mathbf{E}_n\ar[d]^{f_n}\ar[r]&\cdots\ar[r]&\mathbf{E}_1\ar[d]^{f_1}\ar@{->>}[r] &\mathbf{N}\ar@{=}[d]\\
\epsilon':&\mathbf{L}\ \ar@{>->}[r]&\mathbf{E}_n'\ar[r]&\cdots\ar[r]&\mathbf{E}_1'\ar@{->>}[r] &\mathbf{N}.
}
\]
 Moreover, if $f_1,\dots, f_n$ are isomorphisms, we say that $\epsilon$ is isomorphic to $\epsilon'$, and denote it by $\epsilon\cong \epsilon'$.
The relation $\leadsto$  generates an equivalence relation on $\mathbb{E}^n(\mathbf{N},\mathbf{L})$. Namely, two $n$-extensions $\epsilon$ and $\epsilon'$ of $\mathbf{N}$ by $\mathbf{L}$ is {\em equivalent}, if and only if  there exists a chain $\epsilon_0=\epsilon, \epsilon_1,\ldots, \epsilon_k=\epsilon'$ with
\[
\epsilon_0\leadsto \epsilon_1\leftleadsto \epsilon_2\leadsto\cdots\leftleadsto \epsilon_k.
\]
We denote by $[\epsilon]$ the equivalence class of the $n$-extension $\epsilon$, and by $\Ext^n(\mathbf{N},\mathbf{L})$ the set of all equivalence classes of $n$-extensions of $\mathbf{N}$ by $\mathbf{L}$. The set $\Ext^n(\mathbf{N},\mathbf{L})$ is a pointed set with
\[0=[\xymatrix{\mathbf{L}\ar@{=}[r]&\mathbf{L}\ar[r]&\mathbf{0}\ar[r]&\cdots\ar[r]&\mathbf{0}\ar[r]&\mathbf{N}\ar@{=}[r]&\mathbf{N}}]
\] for $n\geq 2$ and 
\[\xymatrix{0=[\mathbf{L}\ \ar@{>->}[r]^-{\tiny\begin{bmatrix}1_{\mathbf{L}}&0\end{bmatrix}}&\mathbf{L}\oplus \mathbf{N}\ar@{->>}[r]^-{\tiny\begin{bmatrix}0\\ 1_{\mathbf{N}}\end{bmatrix}}&\mathbf{N}]}
\]
for $n=1$. 
The following is obvious, see \cite{FRY24}*{Lemma 2.4} or \cite{FYZ25}*{Lemma 2.3}.
\begin{lemma}\label{lem:equiv=0}
 For any $n\geq 1$ and $\mathbf{L}\in \opname{rep}(Q,\mathbb{F}_1)$, if $\Ext^n(\mathbf{L},-)=0$, then $\Ext^{n+1}(\mathbf{L},-)=0$.
\end{lemma}
The following is a direct consequence of Lemma \ref{lem:equiv=0}.
\begin{corollary}
    If $\Ext^n(-,-)\neq 0$ and $\Ext^{n+1}(-,-)=0$ for $\opname{rep}(Q,\mathbb{F}_1)$, then $\opname{gldim} \opname{rep}(Q,\mathbb{F}_1)=n$.
\end{corollary}





\section{Upper bound of global dimensions}\label{s:upper-bound}

Let us first recall a well-known fact in  derived categories of representations of quivers. Let $Q$ be a finite quiver and $k$ a field. Denote by $\der^b(kQ)$ the derived category of right $kQ$-modules with suspension functor $\Sigma$.  For any morphism $L\xrightarrow{f}M$ of right $kQ$-modules, it is well known that $f$ fits into a triangle in $\mathcal{D}^b(kQ)$
    \[
    L\xrightarrow{f}M\to \coker f\oplus \Sigma \ker f\to \Sigma L.
    \]
    By the octahedral axiom, we obtain the following commutative diagram of triangles
    \[
    \xymatrix{L\ar@{=}[d]\ar[r]^\varphi & W\ar[d]^{\psi}\ar[r]^h&\coker f\ar[d]^{\tiny \begin{bmatrix}
        1\\0
    \end{bmatrix}}\ar[r] &\Sigma \ar@{=}[d]\\
    L\ar[r]^f & M\ar[r]&\coker f\oplus \Sigma \ker f\ar[r] &\Sigma L. }
    \]
    Since $L$ and $\coker f$ are $kQ$-modules, we conclude that $W$ is also a $kQ$-module and \[{L}\overset{\varphi}{\rightarrowtail}\ {W}\overset{h}{\twoheadrightarrow}{N}\] is a short exact sequence. In particular, we have a commutative diagram of $kQ$-modules:
    \[
    \xymatrix{L\ar@{=}[d]\ar[r]^{\varphi}&W\ar[d]^{\psi}\ar[r]^h &\coker f\ar@{=}[d]\\ L\ar[r]^f&M\ar[r]&\coker f.}
    \]

We show that a non-additive version of the above result holds  in $\opname{rep}(Q,\mathbb{F}_1)$.

\begin{lemma}\label{lem:pull-back-rep}
Let $Q$ be any quiver, and \[\mathbf{L}\xrightarrow{f}\mathbf{M}\overset{g}{\twoheadrightarrow} \mathbf{N}\]  a sequence of morphisms in $\opname{rep}(Q,\mathbb{F}_1)$ which is exact at $\mathbf{M}$ and $\mathbf{N}$. There exists  a short exact sequence \[\mathbf{L}\overset{\varphi}{\rightarrowtail}\ _f\mathbf{W}\overset{h}{\twoheadrightarrow}\mathbf{N}\] and a homomorphism $\psi:\ _f\mathbf{W}\to \mathbf{M}$ such that the following diagram commutes:
    \begin{equation}\label{diag:key-lemma}
    \xymatrix{\mathbf{L}\ \ar@{=}[d]\ar@{>->}[r]^{\varphi}&_f\mathbf{W}\ar[d]^{\psi}\ar@{->>}[r]^h&\mathbf{N}\ar@{=}[d]\\
    \mathbf{L}\ar[r]^f&\mathbf{M}\ar@{->>}[r]^g&\mathbf{N}.}
    \end{equation}
    Moreover, if $\mathbf{L},\mathbf{M}$ and $\mathbf{N}$ are nilpotent, then so is $ _f\mathbf{W}$. 
\end{lemma}
\begin{proof}
Assume that $\mathbf{L}=(L_i,L_\alpha)_{i\in Q_0,\alpha\in Q_1}$, $\mathbf{M}=(M_i,M_\alpha)_{i\in Q_0,\alpha\in Q_1}$ and $\mathbf{N}=(N_i,N_\alpha)_{i\in Q_0,\alpha\in Q_1}$.

Let us first define the representation $_f\mathbf{W}:=(W_i,W_\alpha)_{i\in Q_0,\alpha\in Q_1}$. For each vertex $i\in Q_0$, let $W_i=L_i\oplus N_i$. For each arrow $\alpha:i\to j$, we define a map $W_\alpha:W_i\to W_j$ as  follows:
\begin{itemize}
    \item if $x\in L_i$, $W_\alpha(x)=L_\alpha(x)$;
    \item if $0\neq y\in N_i$,
    \[
    W_\alpha(y)=\begin{cases}
        0 &\text{if $M_\alpha(g_i^t(y))=0$};\\
        N_\alpha(y)&\text{if $M_\alpha(g_i^t(y))\neq 0$ and $g_j(M_\alpha(g_i^t(y)))\neq 0$};\\
        f_j^t(M_\alpha(g_i^t(y)))&\text{if $M_\alpha(g_i^t(y))\neq 0$ and $g_j(M_\alpha(g_i^t(y)))=0$}. 
    \end{cases}
    \]
\end{itemize}
The following commutative diagram of $\mathbb{F}_1$-linear maps is helpful to understand the definition of $W_\alpha$:
\begin{equation}\label{diag:f-g-morphism}
\xymatrix{\llap{$x\in$ }L_i\ar[d]^{L_\alpha}\ar[r]^{f_i}&M_i\ar[d]^{M_\alpha}\ar@{->>}[r]^{g_i}&N_i\ar[d]^{N_\alpha}\rlap{$\ni y$}\\
L_j\ar[r]^{f_j}&M_j\ar@{->>}[r]^{g_j}&N_j.}
\end{equation}
We claim that $W_\alpha$ is $\mathbb{F}_1$-linear, hence completing the proof that \[_f\mathbf{W}=(W_i,W_\alpha)_{i\in Q_0,\alpha\in Q_1}\]is a representation. By definition, it remains to show that $W_{\alpha}(u)=W_\alpha(v)\neq 0$ implies that $u=v\in W_i$.
\begin{itemize}
    \item[(1)] If both $u,v\in L_i$, then $W_\alpha(u)=L_\alpha(u)=L_\alpha(v)=W_\alpha(v)$. It follows that $u=v$ since $L_\alpha$ is $\mathbb{F}_1$-linear.
    \item[(2)] Assume that $u,v\in N_i$. Since $W_\alpha(u)=W_\alpha(v)\neq 0$, we deduce that $M_\alpha(g_i^t(u))\neq 0\neq M_\alpha(g_i^t(v))$. If follows that $N_\alpha(u)=N_{\alpha}(v)\in N_j$ or $f_j^t(M_\alpha(g_i^t(u)))=f_j^t(M_\alpha(g_i^t(v)))\in L_j$. In either case, we conclude that $u=v$.
    \item[(3)] Finally, assume that $u\in L_i$ and $v\in N_i$. By definition, $W_\alpha(u)=L_\alpha(u)=W_\alpha(v)\neq 0$. It follows that $M_\alpha(g_i^t(v))\neq 0$, $g_j(M_\alpha(g_i^t(v)))=0$ and \[W_\alpha(v)=f_j^t(M_\alpha(g_i^t(v)))=L_\alpha(u).\]
    Applying $f_j$, we obtain $0\neq M_\alpha(g_i^t(v))=f_j(L_\alpha(u))=M_\alpha(f_i(u))$. Consequently, $g_i^t(v)=f_i(u)$ since $M_\alpha$ is $\mathbb{F}_1$-linear. By further applying $g_i$, we obtain \[0\neq v=g_if_i(u)=0,\] a contradiction.
\end{itemize}
This completes the verification of $\mathbb{F}_1$-linearity of $W_\alpha$.

Now we turn to define the homomorphism $\varphi=(\varphi_i)_{i\in Q_0}:\mathbf{L}\rightarrowtail\ _f\mathbf{W}$ and $\psi=(\psi_i)_{i\in Q_0}:\  _f\mathbf{W}\to \mathbf{M}$. 

For each $i\in Q_0$, define $\varphi_i:L_i\rightarrowtail W_i=L_i\oplus N_i$ by $x\mapsto x$. Clearly, $\varphi_i$ is $\mathbb{F}_1$-linear. For any $\alpha:i\to j\in Q_1$, \[W_\alpha(\varphi_i(x))=W_\alpha(x)=L_\alpha(x)=\varphi_j(L_\alpha(x)).\] In particular, the following diagram is commutative:
\[
\xymatrix{L_i\ar[d]^{\varphi_i}\ar[r]^{L_\alpha}&L_j\ar[d]^{\varphi_j}\\ W_i\ar[r]^{W_\alpha}&W_j.}
\]
It follows that $\varphi=(\varphi_i)_{i\in Q_0}:\mathbf{L}\rightarrowtail\ _f\mathbf{W}$ is an injective homomorphism.

For each  $i\in Q_0$, define $\psi_i:W_i=L_i\oplus N_i\to M_i$ as follows
\begin{itemize}
    \item for $x\in L_i$, we define $\psi_i(x):=f_i(x)$;
    \item for $0\neq y\in N_i$, we define $\psi_i(y)=g_i^t(y)$.
\end{itemize}
It is straightforward to check that $\psi_i$ is $\mathbb{F}_1$-linear. Let us consider the following diagram
\begin{equation}\label{diag:psi}
\xymatrix{L_i\oplus N_i=W_i\ar[d]^{\psi_i}\ar[rr]^{W_\alpha}&&W_j\ar[d]^{\psi_j}\\ M_i\ar[rr]^{M_\alpha}&&M_j.}
\end{equation}
For $x\in L_i$, we have \[\psi_j(W_\alpha(x))=\psi_j(L_\alpha(x))=f_j(L_\alpha(x))=M_\alpha(f_i(x))=M_\alpha(\psi_i(x)),\] where the third equality follows from the commutative diagram \eqref{diag:f-g-morphism}.
For $y\in N_i$,
\begin{itemize}
    \item if $M_\alpha(g_i^t(y))=0$, then $W_\alpha(y)=0$. That is, $\psi_j(W_\alpha(y))=0=M_\alpha(\psi_i(y))$.
    \item if $M_\alpha(g_i^t(y))\neq 0$ and $g_j(M_\alpha(g_i^t(y)))\neq 0$, then $N_\alpha(y)\neq 0$ and
    \[
    \psi_j(W_\alpha(y))=\psi_j(N_\alpha(y))=g_j^t(N_\alpha(y))=M_\alpha(g_i^t(y))=M_\alpha(\psi_i(y)),
    \]
    where the third equality follows from the commutative diagram \eqref{diag:f-g-morphism}.
    \item if  $M_\alpha(g_i^t(y))\neq 0$ and $g_j(M_\alpha(g_i^t(y)))= 0$, then
    \[
    \psi_j(W_\alpha(y))=\psi_j(f_j^t(M_\alpha(g_i^t(y))))=f_j(f_j^t(M_\alpha(g_i^t(y))))=M_\alpha(g_i^t(y))=M_\alpha(\psi_i(y)).
    \]
\end{itemize}
In other words, we have verified that the diagram \eqref{diag:psi} is commutative. Hence $\psi: _f\mathbf{W}\to \mathbf{M}$ is a homomorphism. 

Let $h:=g\circ \psi$. It is straightforward to see that $h$ is surjective and $\mathbf{L}\overset{\varphi}{\rightarrowtail}\ _f\mathbf{W}\overset{h}{\twoheadrightarrow}\mathbf{N}$ is a short exact sequence. Furthermore, $f=\psi\circ \varphi$. This completes the proof.

\end{proof}
\begin{remark}
    The representation $_f\mathbf{W}$ in Lemma \ref{lem:pull-back-rep} is unique up to isomorphism. Indeed, since $\mathbf{L}\overset{\varphi}{\rightarrowtail}\ _f\mathbf{W}\overset{h}{\twoheadrightarrow}\mathbf{N}$ is a short exact sequence, we may assume $W_i=L_i\oplus N_i$ for each vertex $i\in Q_0$, $\varphi_i:L_i\rightarrowtail L_i\oplus N_i$ is the canonical inclusion and $h_i:L_i\oplus N_i\twoheadrightarrow N_i$ is the canonical projection.  By analyzing the commutative diagram \eqref{diag:key-lemma},  $W_\alpha$ can only be defined as in the proof.
\end{remark}


Now we are in a position to prove Theorem \ref{thm:main-thm-1}.
\begin{proof}[Proof of Theorem \ref{thm:main-thm-1}]
According to Lemma \ref{lem:equiv=0}, it suffices to show that $\operatorname{Ext}^3(\mathbf{N},\mathbf{L})=0$ for any $\mathbf{L},\mathbf{N}\in \opname{rep}(Q,\mathbb{F}_1)$ or $\opname{rep}(Q,\mathbb{F}_1)_{\rm nil}$.

Let $\epsilon: \mathbf{L}\rightarrowtail \mathbf{U}\xrightarrow{}\mathbf{V}\xrightarrow{f}\mathbf{M}\overset{g}{\twoheadrightarrow}\mathbf{N}$ an exact sequence in $\Ext^3(\mathbf{N},\mathbf{L})$.
Applying Lemma \ref{lem:pull-back-rep}, we obtain the following commutative diagram of exact sequences:
\[
\xymatrix{\mathbf{L}\,\ar@{>->}[r]&\mathbf{U}\ar[r]&\mathbf{V}\ar[r]^f&\mathbf{M}\ar@{->>}[r]^g&\mathbf{N}\\ \mathbf{L}\ar@{=}[d]\ar@{=}[u]\ar@{=}[r]&\mathbf{L}\, \ar@{>->}[u]\ar[r]^0\ar@{=}[d]&\mathbf{V}\ar[d]\ar@{>->}[r]^\varphi\ar@{=}[u]&\mathbf{W}\ar[d]^{g\circ\psi}\ar@{->>}[r]^{g\circ\psi}\ar[u]^{\psi}&\mathbf{N}\ar@{=}[d]\ar@{=}[u]\\
\mathbf{L}\ar@{=}[r]&\mathbf{L}\ar[r]&0\ar[r] &\mathbf{N}\ar@{=}[r]&\mathbf{N}.
}
\]
It follows that $[\epsilon]=[0]$. 

If the exact sequence $\epsilon$ belongs to $\opname{rep}(Q,\mathbb{F}_1)_{\rm nil}$, then all the exact sequences in the above diagram belong to $\opname{rep}(Q,\mathbb{F}_1)_{\rm nil}$. Hence $\Ext^3(\mathbf{N},\mathbf{L})=0$ for $\opname{rep}(Q,\mathbb{F}_1)_{\rm nil}$.
\end{proof}

\section{Classification up to global dimension}\label{s:classification}
\subsection{Embedding/retriction functors}
Let $Q'$ be a subquiver (may not full) of $Q$. 
There are two natural functors associated with $\opname{rep}(Q',\mathbb{F})$ and $\opname{rep}(Q,\mathbb{F})$:
\begin{itemize}
    \item[(1)] The embedding functor \begin{eqnarray*}\iota:&\opname{rep}(Q',\mathbb{F})&\to \opname{rep}(Q,\mathbb{F}),\\
    &\mathbf{V}=(V_i,V_\alpha)&\mapsto \iota(\mathbf{V}):=(W_i,W_\alpha)
    \end{eqnarray*}
    where \[W_i=\begin{cases}
        V_i& \text{if $i\in Q_0'$};\\
        0&\text{otherwise},
    \end{cases}\ \text{and}\  W_\alpha=\begin{cases}
        V_\alpha&\text{if $\alpha\in Q_1'$};\\ 0&\text{otherwise}.
    \end{cases}\]
    For a homomorphism $f=(f_i)_{i\in Q_0'}:\mathbf{V}\to \mathbf{U}$, the homomorphism $\iota(f):=(g_i)_{i\in Q_0}:\iota(\mathbf{V})\to \iota(\mathbf{U})$ is defined as \[g_i=\begin{cases}
        f_i&\text{if $i\in Q_0'$};\\
        0&\text{else}.
    \end{cases}\]
    \item[(2)] The restriction functor
    \begin{eqnarray*}
        \text{Res}: &\opname{rep}(Q,\mathbb{F})&\to \opname{rep}(Q',\mathbb{F}),\\
        &\mathbf{W}=(W_i,W_\alpha)&\mapsto \text{Res}(\mathbf{W}):=(W_i,W_\alpha)_{i\in Q_0',\alpha\in Q'_1}
    \end{eqnarray*}
    For a homomorphism $f=(f_i)_{i\in Q_0}:\mathbf{W}\to \mathbf{W'}$, the homomorphism $\text{Res}(f):=(f_i)_{i\in Q_0'}:\text{Res}(\mathbf{W})\to \text{Res}(\mathbf{W'})$.
\end{itemize}
By definition, both $\iota$ and $\text{Res}$ are exact, i.e., preserve exact sequences, and $\text{Res}\circ \iota=\id$. As a consequence, for any $1\leq i\leq 2$, we obtain  two well-defined  maps:
\[
\iota^i:\mathbb{E}^i(\mathbf{N},\mathbf{L})\to \mathbb{E}^i(\iota(\mathbf{N}),\iota(\mathbf{L})),
\]
\[
\text{Res}^i:\mathbb{E}^i(\mathbf{U},\mathbf{W})\to \mathbb{E}^i(\text{Res}(\mathbf{U}),\text{Res}(\mathbf{W})).
\]
Clearly, for any $\epsilon \in \mathbb{E}^i(\mathbf{N},\mathbf{L})$, $\text{Res}^i\circ \iota^i(\epsilon)=\epsilon$. If $\epsilon,\epsilon'\in \mathbb{E}^i(\mathbf{N},\mathbf{L})$ such that $[\epsilon]=[\epsilon']$, we clearly have $[\iota^i(\epsilon)]=[\iota^i(\epsilon')]$. As a consequence, $\iota^i$ further induces a map
\[
\iota^i:\Ext_{Q'}^i(\mathbf{N},\mathbf{L})\to \Ext_Q^i(\iota(\mathbf{N}),\iota(\mathbf{L})).
\]
Similarly, 

The map $\text{Res}^i$ also induces a map
\[
\text{Res}^i:\Ext_Q^i(\mathbf{U},\mathbf{W})\to \Ext_{Q'}^i(\text{Res}(\mathbf{U}),\text{Res}(\mathbf{W})).
\]
\begin{lemma}\label{lem:embedding-ext}
    The map $\iota^i$ is an injective $\mathbb{F}_1$-linear map.
\end{lemma}
\begin{proof}
    Let $\epsilon,\epsilon'\in \mathbb{E}^i(\mathbf{N},\mathbf{L})$ for $\mathbf{L},\mathbf{N}\in \opname{rep}(Q',\mathbb{F}_1)$. It suffices to show that $[\epsilon]=[\epsilon']\in \Ext^i_{Q'}(\mathbf{N},\mathbf{L})$ if and only if $[\iota^i(\epsilon)]=[\iota^i(\epsilon')]\in \Ext^i_Q(\iota(\mathbf{N}),\iota(\mathbf{L}))$. The direction ``$\Rightarrow$" is clear since $\iota^i:\Ext_{Q'}^i(\mathbf{N},\mathbf{L})\to \Ext_Q^i(\iota(\mathbf{N}),\iota(\mathbf{L}))$ is a map.

    Assume that $[\iota^i(\epsilon)]=[\iota^i(\epsilon')]$. Without loss of generality, we may assume that there exist exact sequences $\epsilon_1,\ldots, \epsilon_m\in \mathbb{E}^i(\iota(\mathbf{N}),\iota(\mathbf{L}))$ such that
    \[
    \iota^i(\epsilon)\leadsto \epsilon_1\leftleadsto \epsilon_2\leadsto\cdots\leadsto\epsilon_m\leftleadsto \iota^i(\epsilon').
    \]
    By applying $\text{Res}^i$, we obtain
    \[
    \epsilon=\text{Res}^i\iota^i(\epsilon)\leadsto \text{Res}^i(\epsilon_1)\leftleadsto \text{Res}^i(\epsilon_2)\leadsto\cdots\leadsto\text{Res}^i(\epsilon_m)\leftleadsto \text{Res}^i(\iota^i(\epsilon'))=\epsilon'.
    \]
    Note that each $\text{Res}^i(\epsilon_j)\in \mathbb{E}^i(\mathbf{N},\mathbf{L})$ for $1\leq j\leq m$. We conclude that $[\epsilon]=[\epsilon']\in\Ext^i_{Q'}(\mathbf{N},\mathbf{L})$.
\end{proof}
\begin{remark}
    The map $\text{ Res}^i:\Ext_Q^i(\mathbf{U},\mathbf{W})\to \Ext_{Q'}^i(\text{Res}(\mathbf{U}),\text{Res}(\mathbf{W}))$ may not be $\mathbb{F}_1$-linear in general. On the other hand, we have $\text{Res}^i\circ \iota^i=\id: \Ext_{Q'}^i(\mathbf{N},\mathbf{L})\to \Ext_{Q'}^i(\mathbf{N},\mathbf{L})$.
\end{remark}
\begin{proposition}\label{prop:subquiver-a3-case}
    Let $Q$ be any quiver. If $Q$ has a subquiver $\xymatrix{\cdot\ar[r]&\cdot\ar[r]&\cdot}$, then $\opname{gldim}\opname{rep}(Q,\mathbb{F}_1)=2$.
\end{proposition}
\begin{proof}
    Let $Q'$ be the subquiver $\xymatrix{\cdot\ar[r]&\cdot\ar[r]&\cdot}$ of $Q$. By \cite{FRY24}*{Theorem 3.10}, $\Ext^2_{Q'}(-,-)\neq 0$. It follows that $\Ext^2_Q(-,-)\neq 0$ by Lemma \ref{lem:embedding-ext}. According to Theorem \ref{thm:main-thm-1}, $\text{gldim} \opname{rep}(Q,\mathbb{F}_1)=2$.
\end{proof}
\subsection{Quivers with oriented cycles}
In this subsection, let $Q$ be a quiver  with an oriented cycle $Q'$. We label the vertex set of $Q'$ by $Q_0'=\{1,\dots, n\}$  such that the arrows are precisely $\{\alpha_i:i\to i+1\}$(taken modulo $n$). Let $\mathbf{S}_i$ be the nilpotent simple representation associated with vertex $i$. Denote by $\Ext^2_{Q'}(-,-)_{\rm nil}$ (resp. $\Ext^2_{Q}(-,-)_{\rm nil}$)  the second Yoneda extension space in the category $\opname{rep}(Q,\mathbb{F}_1)_{\rm nil}$ (resp. $\opname{rep}(Q,\mathbb{F}_1)_{\rm nil}$).
The following has been established in \cite{FYZ25}.
\begin{lemma}\label{lem:2-ext-infinite}
    For any $i,j\in Q_0'$, we have $\Ext^2_{Q'}(\mathbf{S}_i,\mathbf{S}_j)_{\rm nil}\cong \mathbb{N}$, where $\mathbb{N}$ is the set of nonnegative integers as a point set with base point $0$.
\end{lemma}




   
    
\begin{lemma}\label{lem:identify-subquiver-ext}
    For any $i,j\in Q_0'$, we have $\Ext^2_{Q'}(\mathbf{S}_i,\mathbf{S}_j)_{\rm nil}=\Ext^2_{Q'}(\mathbf{S}_i,\mathbf{S}_j)$.
\end{lemma}
\begin{proof}
    Let $\epsilon: \xymatrix{\mathbf{S}_j\ar@{>->}[r]&\mathbf{M}\ar[r]&\mathbf{N}\ar@{->>}[r]&\mathbf{S}_i}$ be any exact sequence in $\opname{rep}(Q',\mathbb{F}_1)$.

    By Krull-Schmidt theorem and Corollary \ref{cor:vanish-hom}, we can rewrite $\epsilon$ as 
    \[\xymatrix{\mathbf{S_j}\,\ar@{>->}[r]^{\tiny\begin{bmatrix}
        l_j\\ 0\end{bmatrix}}&\overline{\mathbf{M}}\oplus \mathbf{L}\ar[r]^{\tiny\begin{bmatrix}\overline{f}&\\ &\id_\mathbf{L}\end{bmatrix}}&\overline{\mathbf{N}}\oplus \mathbf{L}\ar@{->>}[r]^{\tiny\begin{bmatrix}
            0&\pi_i
        \end{bmatrix}}&\mathbf{S}_j,}
    \]
    where $\overline{\mathbf{M}}$ and $\overline{\mathbf{N}}$ belong to $\opname{rep}(Q',\mathbb{F}_1)_{\rm nil}$, while $L$ is either zero or a direct sum of indecomposable non-nilpotent representation of $Q'$. It follows that $\epsilon$ is equivalent to 
   \[\xymatrix{\epsilon':\mathbf{S_j}\,\ar@{>->}[r]^{l_j}&\overline{\mathbf{M}}\ar[r]^{\overline{f}}&\overline{\mathbf{N}}\ar@{->>}[r]^{\pi_i}&\mathbf{S}_j}
    \]
    in $\Ext^2_{Q'}(\mathbf{S}_i,\mathbf{S}_j)_{\rm nil}$. On the other hand, for any $[\epsilon_1],[\epsilon_2]\in \Ext^2_{Q'}(\mathbf{S}_i,\mathbf{S}_j)_{\rm nil}$, if $[\epsilon_1]=[\epsilon_2]\in \Ext^2_{Q'}(\mathbf{S}_i,\mathbf{S}_j)$, we conclude that $[\epsilon_1]=[\epsilon_2]\in \Ext^2_{Q'}(\mathbf{S}_i,\mathbf{S}_j)_{\rm nil}$  by the above discussion and Corollary \ref{cor:vanish-hom}. Hence, $\Ext^2_{Q'}(\mathbf{S}_i,\mathbf{S}_j)_{\rm nil}=\Ext^2_{Q'}(\mathbf{S}_i,\mathbf{S}_j)$.
\end{proof}
\begin{proposition}\label{prop:cylce-cases}
   Let $Q$ be a quiver with an oriented cycle, then $\opname{gldim} \opname{rep}(Q,\mathbb{F}_1)=2$ and  $\opname{gldim} \opname{rep}(Q,\mathbb{F}_1)_{\rm nil}=2$.
\end{proposition}
\begin{proof}
    Let $Q'$ be an oriented cycle of $Q$. By Lemma \ref{lem:identify-subquiver-ext}, Lemma \ref{lem:2-ext-infinite} and Lemma \ref{lem:embedding-ext}, we conclude that $\Ext^2_Q(-,-)\neq 0$. Therefore $\text{gldim} \opname{rep}(Q,\mathbb{F}_1)=2$ by Theorem \ref{thm:main-thm-1}. 

For the statement $\opname{gldim} \opname{rep}(Q,\mathbb{F}_1)_{\rm nil}=2$, it is straightforward to see that the embedding functor $\iota$ and the restriction $\text{Res}$ preserve nilpotent representations. In particular, Lemma \ref{lem:embedding-ext} holds when restricting to the full subcategory of nilpotent representations. Hence $\opname{gldim} \opname{rep}(Q,\mathbb{F}_1)_{\rm nil}=2$.
    
\end{proof}

\subsection{Bipartite quivers}\label{ss:bipartie-case}
Recall that a quiver is called \emph{bipartite} if every vertex is either a source or a sink. Equivalently, a quiver is bipartite if and only if it contains no oriented cycles and no subquiver of type $\mathbb{A}_3$ with a linear orientation.

The aim of this subsection is to prove the following.
\begin{theorem}\label{thm:bipartite-case}
    Let $Q$ be a bipartite quiver. Then $\opname{gldim} \opname{rep}(Q,\mathbb{F}_1)\leq 1$.
\end{theorem}
\begin{proof}
    Let \[\xymatrix{\epsilon: \mathbf{K}\,\ar@{>->}[r]^{\kappa}&\mathbf{L}\ar@{->>}[dr]^{f}\ar[rr]&&\mathbf{M}\ar@{->>}[r]^g&\mathbf{N}\\ && \mathbf{V}\ar@{>->}[ur]^\beta}\] be an exact sequence in $\mathbb{E}^2(\mathbf{N},\mathbf{K})$, where the middle morphism factors as $\beta\circ f$. It suffices to show that $[\epsilon]=[0]\in \Ext^2(\mathbf{N},\mathbf{K})$.

We first introduce two representations associated to $\epsilon$.
Define a subrepresentation $\mathbf{U}=(U_i,U_\alpha)$ of $\mathbf{V}=(V_i,V_\alpha)$ as follows:
\[
U_i=\begin{cases}
    0&\text{if $\exists\ i\to j$ such that $K_j\neq 0$};\\
    V_i&\text{else}.
\end{cases}\quad \text{and}\quad U_\alpha=V_\alpha|_{U_s(\alpha)}.
\]
Equivalently, $U_i=V_i$ if $i$ is a sink or $i$ is a source but for any arrow $\alpha:i\to j$ starting from $i$, we have $K_j=0$. For any $\alpha:i\to j$, since $Q$ is bipartite, $j$ is a sink, we have $U_j=V_j$; while $U_i$ is either zero or $V_i$, i.e., a subspace of $V_i$. Thus $U_\alpha$ is well-defined. This shows that $\mathbf{U}$ is a representation of $Q$. It is straightforward to see that $\mathbf{U}$ is a subrepresentation of $\mathbf{V}$.

Define a subrepresentation $\mathbf{W}=(W_i,W_\alpha)$ of $\mathbf{M}=(M_i,M_\alpha)$ as follows:
\[
W_i=\begin{cases}
    g_i^t(N_i)&\text{if $\exists\ i\to j$ such that $K_j\neq 0$};\\
    M_i&\text{else}.
\end{cases} \quad \text{and}\quad W_\alpha:=M_\alpha|_{W_{s(\alpha)}}.
\]
Again, for each $i\in Q_0$, $W_i$ is a subspace of $M_i$, hence $W_\alpha$ is well-defined and $\mathbf{W}$ is a representation of $Q$. It is also obvious to see that $\mathbf{W}$ is a subrepresentation of $\mathbf{M}$.

Let us  denote the embeddings by $\delta:\mathbf{U}\rightarrowtail \mathbf{V}$ and $\gamma:\mathbf{W}\rightarrowtail \mathbf{M}$. We claim that $\im\beta\circ\delta\subseteq \mathbf{W}$. Since $\beta\circ\delta$ is a homomorphism of representations, it suffices to check that $\beta_i\circ\delta_i(U_i)\subseteq W_i$ for any $i\in Q_0$. There is nothing to prove if $U_i=0$. Assume $U_i\neq 0$, by definition of $\mathbf{U}$, $i$ is either a sink  or a source but for any $\alpha:i\to j$, one has $K_j=0$. In either cases, we have $W_i=M_i$ and hence $\beta_i\circ\delta_i(U_i)\subseteq M_i$. 
As a consequence, we obtain a homomorpshim $\beta\circ \delta:\mathbf{U}\to \mathbf{W}$.

 We claim that
\begin{equation}\label{diag:exact-seq}
\xymatrix{
U \ar@{>->}[r]^{\beta\circ \delta} & \mathbf{W} \ar@{->>}[r]^{g\circ \gamma} & \mathbf{N}
}   
\end{equation}
is a short exact sequence. 

The homomorphism $\beta \circ \delta$ is injective because both $\delta$ and $\beta$ are injective. Next, recall that for each vertex $i$, we have either $W_i = M_i$ or $W_i = g_i^t(N_i)$; in either case, it follows that $(g \circ \gamma)_i(W_i) = N_i$. Hence, $g \circ \gamma$ is surjective. 
It remains to show that $\ker (g \circ \gamma) = \im (\beta \circ \delta)$. The inclusion $\im (\beta \circ \delta) \subseteq \ker (g \circ \gamma)$ follows immediately from the relation $g \circ \beta = 0$. To establish the reverse inclusion, we check the condition component-wise at each vertex $i$:
\begin{itemize}
    \item \textbf{Case 1:} Suppose $i$ is a source and there exists an arrow $i \to j$ such that $K_j \neq 0$. Then $W_i = g_i^t(N_i)$, which implies that $\ker(g \circ \gamma)_i = \{0\} \subseteq \im(\beta \circ \delta)_i$.
    \item \textbf{Case 2:} Suppose $i$ is a sink, or $i$ is a source such that $K_j = 0$ for all arrows $i \to j$. In both situations, we have $W_i = M_i$ and $U_i = V_i$. Then it follows from the exactness of the ambient sequence that $\ker(g \circ \gamma)_i \subseteq \im(\beta \circ \delta)_i$.
\end{itemize}
This completes the proof that \eqref{diag:exact-seq} is a short exact sequence.

For each vertex $i\in Q_0$, we define an $\mathbb{F}_1$-linear map $u_i:U_i\to L_i$ as follows:
\begin{itemize}
    \item if $U_i=0$, then $u_i=0$;
    \item if $U_i\neq 0$, then $U_i=V_i$. We define $u_i=f_i^t:V_i\to L_i$.
\end{itemize}
For any arrow $\alpha:i\to j$, we consider the following diagram
\begin{equation}\label{diag:morphism-comm}
\xymatrix{U_i\ar[d]^{u_i}\ar[r]^{U_\alpha}& U_j\ar[d]^{u_j}\\ L_i\ar[r]^{L_\alpha}&L_j.}
\end{equation}

If $U_i=\{0\}$, the diagram \eqref{diag:morphism-comm} is clearly commutative.  Assume $U_i\neq 0$. If follows that $U_i=V_i$ and $K_j=0$. Consequently, $f_j:L_j\to V_j$ is an isomorphism. It follows that for any $0\neq x\in V_i$, \[
V_\alpha(x)=V_\alpha(f_i(f_i^t(x))=f_j(L_\alpha(f_i^t(x))).
\]
Since $f_j$ is an isomorphism, we obtain $f_j^t\circ V_\alpha(x)=L_\alpha\circ f_i^t(x)$. Hence, \eqref{diag:morphism-comm} is commutative. In particular, $u=(u_i):\mathbf{U}\to \mathbf{L}$ is a homomorphism. Furthermore, we have $f\circ u=\delta$.

Finally, we claim that $\mathbf{K}\oplus \mathbf{U}\xrightarrow{\tiny\begin{bmatrix}
    \kappa& u
\end{bmatrix}}\mathbf{L}$ is a homomorpohism of representations. We have to show that $\begin{bmatrix}
    \kappa_i& u_i
\end{bmatrix}$ is $\mathbb{F}_1$-linear for each vertex $i\in Q_0$, and for each arrow $\alpha:i\to j$, the following diagram is commutative
\[
\xymatrix{K_i\oplus U_i\ar[d]_{\tiny \begin{bmatrix}
    K_\alpha&\\ &U_\alpha
\end{bmatrix}}\ar[r]^-{\tiny\begin{bmatrix}
    \kappa_i& u_i
\end{bmatrix}}&L_i\ar[d]^{L_\alpha}\\ K_j\oplus U_j\ar[r]^-{\tiny \begin{bmatrix}
    \kappa_j& u_j
\end{bmatrix}}&L_j.}
\]
The commutativity of the diagram follows that $\kappa$ and $u$ are homomorphism once we verified that $\begin{bmatrix}
    \kappa_i& u_i
\end{bmatrix}$ is $\mathbb{F}_1$-linear.
If $U_i=0$, then $\begin{bmatrix}
    \kappa_i& u_i
\end{bmatrix}=\kappa_i$ is $\mathbb{F}_1$-linear. Now assume that $U_i\neq 0$. It follows that $U_i=V_i$ and $u_i=f_i^t:V_i\to L_i$. Note that $\im \kappa_i=\ker f_i$. We conclude that $\im \kappa_i\cap \im u_i=\{0\}$. Therefore $\begin{bmatrix}
    \kappa_i& u_i
\end{bmatrix}$ is $\mathbb{F}_1$-linear.

Putting all of these together, we obtain the following commutative diagram of exact sequences:
\[
\xymatrix{\mathbf{K}\ar@{=}[d]\ar@{>->}[r]^1&\mathbf{K}\ar[rr]^0&&\mathbf{N}\ar@{->>}[r]^1&\mathbf{N}\ar@{=}[d]\\
\mathbf{K}\ar@{>->}[r]^{\tiny \begin{bmatrix}
    1\\0
\end{bmatrix}}&\mathbf{K}\oplus \mathbf{U}\ar[u]_{\tiny\begin{bmatrix}
    1&0
\end{bmatrix} }\ar[rr]_{\tiny \begin{bmatrix}
    0&\beta\circ\delta
\end{bmatrix}}\ar[d]_{\tiny \begin{bmatrix}
    \kappa&u
\end{bmatrix}}&&\mathbf{W}\ar[u]^{g\circ \gamma}\,\ar@{>->}[d]^{\gamma}\ar@{->>}[r]^{g\circ \gamma}&\mathbf{N}\\
\mathbf{K}\ar@{=}[u]\ar@{>->}[r]^{\kappa}&\mathbf{L}\ar[rr]^{\beta\circ f}&&\mathbf{M}\ar@{->>}[r]^g&\mathbf{N}.\ar@{=}[u]
}
\]
Therefore $[\epsilon]=[0]$. This finishes the proof.
\end{proof}
\subsection{Proof of Theorem \ref{thm:main-thm-2}}\label{ss:proof-thm-2}
The statement $(3)$ of Theorem \ref{thm:main-thm-2} follows from Proposition \ref{prop:subquiver-a3-case} and Proposition \ref{prop:cylce-cases}.

Now let $Q$ be a bipartite quiver. If $|Q_0|>1$, then $Q$ has a subquiver $Q'$ of type $\mathbb{A}_2$. It has been proved that $\Ext_{Q'}^1(-,-)\neq 0$ by \cite{FRY24}*{Remark 3.7}. We conclude that $\opname{gldim}\opname{rep}(Q,\mathbb{F}_1)=1$ by Lemma \ref{lem:embedding-ext} and Theorem \ref{thm:bipartite-case}. 

If $Q$ consists of a single vertex, it has been computed by \cite{FRY24}*{Remark 3.7} that $\Ext^1_Q(-,-)=0$. Consequently, $\opname{gldim}\opname{rep}(Q,\mathbb{F}_1)=0$. This finishes the proof of Theorem \ref{thm:main-thm-2}.

\subsection*{Acknowledgements}
This work was supported by the National Natural Science Foundation of China (Grant No. 12571040).
The author would like to acknowledge the assistance of AI language models, specifically ChatGPT 5.4 and Opus 4.6, for providing conceptual inspiration and fruitful dialogue during the preparation of this manuscript. In particular, the core idea leading to the formulation of Lemma \ref{lem:pull-back-rep} emerged during discussions with these tools. The author remains solely responsible for the correctness of the mathematical proofs and the final content.

\bibliographystyle{plain}
\bibliography{ref}
\end{document}